\documentclass[11pt]{article}
\usepackage[T1]{fontenc}
\usepackage{lmodern}
\usepackage[margin=1in]{geometry}
\usepackage{amsmath,amssymb,amsthm,mathtools,mathrsfs}
\usepackage{microtype}
\usepackage{enumitem}
\usepackage[colorlinks=true,linkcolor=blue,citecolor=blue,urlcolor=blue]{hyperref}
\numberwithin{equation}{section}
\newtheorem{theorem}{Theorem}[section]
\newtheorem{proposition}[theorem]{Proposition}
\newtheorem{lemma}[theorem]{Lemma}
\newtheorem{corollary}[theorem]{Corollary}
\theoremstyle{definition}

\theoremstyle{remark}
\newtheorem{remark}[theorem]{Remark}
\newcommand{\dd}{\mathrm d}
\newcommand{\ii}{\mathrm i}
\newcommand{\R}{\mathbb R}
\newcommand{\C}{\mathbb C}
\newcommand{\Z}{\mathbb Z}
\newcommand{\BV}{\operatorname{BV}}

\newcommand{\Tr}{\operatorname{Tr}}
\newcommand{\cH}{\mathcal H}

\hypersetup{pdftitle={Constructing Two Metrics for Hodge Theory of Constrained Bundles: Harmonic Stability and Magnetic Coercivity},pdfauthor={Dongzhe Zheng}}

\title{Constructing Two Metrics for Hodge Theory\newline
of Constrained Bundles:\newline
Harmonic Stability and Magnetic Coercivity}
\author{Dongzhe Zheng}
\date{}

\begin{document}
\maketitle

\begin{abstract}
We study two metric constructions for Hodge theory on bundles, together
with the response of harmonic representatives and magnetic energies to
degeneration. Our first result is an $L^1$ coercivity estimate for the
covariant derivative on functions of bounded variation: the coercivity
constant is positive exactly when the connection has no nonzero parallel
section. Extension across internal interfaces then proves the uniform
magnetic frustration conjecture of Chakradhar, Gittins, Habib, and
Peyerimhoff, and yields positivity of both magnetic Cheeger constants on
connected manifolds. For weighted de Rham complexes, we relate the defect
in cohomological energy to changes in harmonic representatives.
Bracket-generating harmonic vector fields give a quantitative recovery
estimate with no lower bound on the varying positive weight. On the
Heisenberg nilmanifold, the resulting $H^{-1/2}$ estimate has optimal
Sobolev exponent. For circle bundles over surfaces, two connection metrics
have identical cohomological Hodge pairings. A curvature upper bound gives
a uniform positive Hodge spectral gap even as the adapted metrics lose
rank. We identify the scalar spectral limit in every circle character;
zero-density regions contribute magnetic Dirichlet-to-Neumann energies
that retain their flat holonomy.
\end{abstract}


\section{Introduction}

A connection and an inner product on differential forms determine
several complementary geometric quantities. Harmonic representatives
describe cohomology through energy minimization. The positive spectrum
measures coercivity away from these representatives. Parallel sections
record the joint obstruction imposed by curvature and holonomy. Comparing
these quantities under a change of metric requires control of both the
differential complex and the norms placed on it.

We consider two such constructions. A parallel adjoint line in a principal
bundle carries a de Rham complex with geometric scalar weights. On a
circle bundle over a surface, the invariant de Rham differential gives
a second, explicit complex, and the curvature supplies an adapted metric.
These models lead to quantitative questions about recovery of weights,
persistence of spectral gaps, and the contribution of a region whose
limiting volume density vanishes. The magnetic covariant derivative
arising in the circle characters also leads to a domain-uniform
coercivity theorem.

\subsection{Magnetic coercivity and a published conjecture}

For a compact Riemannian manifold $M$ with smooth boundary and a smooth
real one-form $A$, write $\dd_A=\dd+\ii A$ and
\[
 \mathcal B_M=\{\dd q/(\ii q):q\in C^\infty(M,S^1)\}.
\]
The condition $A\notin\mathcal B_M$ expresses the obstruction to a
global unit parallel phase. For a domain $D$, its magnetic frustration is
\[
 \iota_A(D)=\inf_{\tau\in C^\infty(D,S^1)}
                    \int_D|\dd_A\tau|\,\dd V_g.
\]
Lange, Liu, Peyerimhoff, and Post developed the continuum frustration
index and magnetic Cheeger constants on closed manifolds
\cite[Definitions~7.2--7.3]{LLPP2015}. Jammes established a
two-constant Cheeger inequality for the first positive nonmagnetic
Steklov eigenvalue \cite[Theorem~1]{Jammes2015}.
Chakradhar, Gittins, Habib, and Peyerimhoff introduced the corresponding
boundary magnetic constants and proved a bound for the lowest magnetic
Steklov eigenvalue \cite[Definition~2.4 and Theorem~2.5]{CGHP}.
Their conjecture in Section~4.1
asks whether $\iota_A(D)$ has a uniform positive lower bound when $D$
occupies almost all of $M$ and its internal boundary has sufficiently
small area, assuming $A\notin\mathcal B_M$.

Theorem~\ref{mag:conjecture} proves this conjecture with its stated
domain quantifiers. The argument follows from a stronger estimate.
On a connected manifold with $A\notin\mathcal B_M$, there is
$\kappa_A>0$ such that
\begin{equation}
 |Du+\ii Au\,\dd V_g|(M)\ge\kappa_A\|u\|_{L^1(M)}
 \qquad (u\in\BV(M;\C)).
\end{equation}
The compactness input belongs to the magnetic BV theory developed by
Pinamonti, Squassina, and Vecchi \cite{PSV2019} and the classical BV
theory in \cite{AFP}. Normalized functions violating this estimate
would have an $L^1$ limit
which is a nonzero global parallel section. Extending a unit phase by
zero across the internal boundary of $D$ gives
\begin{equation}\label{intro:domain}
 \iota_A(D)+2\operatorname{Area}(\partial D\cap\operatorname{int}M)
 \ge\kappa_A\operatorname{Vol}(D).
\end{equation}
The coefficient two allows two phase traces along an internal slit.
Theorem~\ref{mag:cheeger-positive} also proves strict positivity of both
magnetic Cheeger constants. The conjecture itself holds on disconnected
manifolds; the coercivity criterion there is absence of a nonzero parallel
section on every component. This distinction is included in the proof.

\subsection{Harmonic response and metric degeneration}

For a positive weight $\rho$ on a closed manifold, let $E_\rho(c)$ be
the least weighted energy of a representative of $c\in H^1(M;\R)$.
Fix reference harmonic forms $h_j$ and their dual vector fields $X_j$.
The identity
\[
 \int_M\rho\sum_j|h_j|^2\,\dd V_g-\sum_j E_\rho([h_j])
 =\sum_j\|h_j-h_\rho([h_j])\|_\rho^2
\]
relates a scalar energy defect to the change in harmonic representatives.
When brackets of the $X_j$ of length at most $r$ span the tangent bundle,
Theorem~\ref{resp:bracket-estimate} controls the normalized weight ratio
in $H^{-1+1/r}$. Its constant uses the reference data and an upper bound
on that ratio. On the volume-one Heisenberg nilmanifold, the energy of
the two horizontal cohomology classes has a unique maximizer at
$\rho=1$. Theorem~\ref{resp:sharp-contact} proves the associated
$H^{-1/2}$ deficit estimate and its optimality through explicit central
Fourier modes.

The minimum principle and the displayed energy-defect identity are
specializations of the classical effective-operator framework
\cite[Section~13.1]{Milton2002}. The Hilbert-space exposition of Welters
\cite[Theorem~6 and Corollary~8]{Welters2025} applies directly to the
exact, harmonic, and coexact Hodge decomposition and includes concavity.
The additional
conclusions here are the bracket-generating stability estimate and its
sharp geometric realization. The subelliptic input is the theorem of
Rothschild and Stein \cite{RothschildStein1976}.

For a circle bundle $P\to\Sigma$, write
$\dd\alpha=\pi^*(f\,\dd A_h)$ with $f>0$. We compare
\[
 g_A=\alpha^2+\pi^*(\bar f h),\qquad
 g_B=\alpha^2+\pi^*(fh),\qquad
 \bar f=|\Sigma|_h^{-1}\int_\Sigma f\,\dd A_h.
\]
Proposition~\ref{geo:harmonic} identifies their equal Hodge pairings on
cohomology. Theorem~\ref{geo:gap} gives a positive spectral gap in every
degree using only an upper bound for $f$ and the fixed base metric.
The connection-metric spectral structure is developed by Nagy
\cite{Nagy1999}; our direct estimates specify constants that persist
when $\inf f$ tends to zero. Theorem~\ref{lim:scalar} identifies the
entire scalar limit through its circle characters. A zero-density region
contributes its magnetic Dirichlet-to-Neumann form to the limit energy.
The explicit annulus operator exhibits the retained flat holonomy;
its eigenvalue branches were calculated in
\cite{CekicSiffert,ProvenzanoSavo}.
On such a region, \eqref{intro:domain} also controls the cost of removing
the holonomy obstruction by internal cuts.

All differentials, adjoints, boundary conventions, and limiting form
domains are defined below. The magnetic results use a fixed smooth
background metric. The degeneration results concern the specified
family of connection metrics and its scalar spectral limit.

\section{Magnetic BV coercivity and frustration}
\label{sec:magnetic}

Let $(M,g)$ be a compact smooth Riemannian manifold with smooth boundary,
and let $A\in\Omega^1(M;\R)$ be smooth. All measures and norms in this
section use the fixed metric $g$. The magnetic differential on complex
functions is
\[
 \dd_Au=\dd u+\ii Au.
\]
The set of potentials which can be removed by a unitary gauge is
\[
 \mathcal B_M
 =\left\{\frac{\dd q}{\ii q}:q\in C^\infty(M,S^1)\right\}.
\]
Thus $A\in\mathcal B_M$ precisely when $\dd_A$ has a smooth unit-modulus
parallel section. Indeed, if $A=\dd q/(\ii q)$, then
$\dd_A\overline q=0$, and the converse follows by conjugation.
The corresponding gauge and zero-eigenvalue criterion for closed
manifolds is classical; see Shigekawa \cite{Shigekawa1987} and its
formulation in \cite[Theorem~7.1]{LLPP2015}.

For a smooth domain $D\subset M$, write
\[
 P_M(D)=\operatorname{Area}(\partial D\cap\operatorname{int}M),
 \qquad
 E_M(D)=\operatorname{Area}(\partial D\cap\partial M).
\]
Following \cite[Definition 2.4]{CGHP}, admissible domains are nonempty
open subsets with compact closure whose boundary is a smooth hypersurface.
The estimates below also allow an internal smooth hypersurface with
$D$ on both sides. Following \cite[Definition~7.2]{LLPP2015},
the magnetic frustration is
\[
 \iota_A(D)
 =\inf_{\tau\in C^\infty(D,S^1)}
       \int_D|\dd_A\tau|\,\dd V_g.
\]
The boundary magnetic Cheeger constants of
\cite[Definition~2.4]{CGHP} are
\begin{equation}
 \label{mag:cheeger-def}
 h_A(M)=\inf_D\frac{\iota_A(D)+P_M(D)}{\operatorname{Vol}(D)},
 \qquad
 h'_A(M)=\inf_{\substack{D\\E_M(D)>0}}
          \frac{\iota_A(D)+P_M(D)}{E_M(D)}.
\end{equation}
We assume $\partial M\ne\varnothing$ when using $h'_A(M)$.

We prove the uniform frustration conjecture in
\cite[Section 4.1]{CGHP}. We also prove positivity of both constants
in \eqref{mag:cheeger-def} on connected manifolds whenever
$A\notin\mathcal B_M$. The main estimate controls the covariant
variation of a function, including the jumps produced by internal cuts.

\subsection{Covariant variation and its kernel}

For $u\in\BV(M;\C)$, let $Du$ denote its distributional differential,
a complex cotangent-valued Radon measure. Derivatives are relative to
the interior of $M$, so $Du$ contains no jump along $\partial M$.
Define
\[
 \mathscr D_Au=Du+\ii Au\,\dd V_g,
 \qquad
 \kappa_A(M)
 =\inf_{\substack{u\in\BV(M;\C)\\u\ne0}}
       \frac{|\mathscr D_Au|(M)}{\|u\|_{L^1(M)}}.
\]
Here $|\mathscr D_Au|$ is the total variation measure. The elementary
estimate
\begin{equation}
 \label{mag:variation-control}
 |Du|(M)\le|\mathscr D_Au|(M)
                 +\|A\|_{L^\infty(M)}\|u\|_{L^1(M)}
\end{equation}
allows us to use ordinary BV compactness. We use the standard compact
embedding $\BV(M)\hookrightarrow L^1(M)$ on bounded sets and the BV
trace theorem; their manifold versions follow in finitely many smooth
coordinate charts from the corresponding bounded-domain results
\cite{AFP}.

Magnetic BV spaces, their measure representation, norm equivalence with
ordinary BV for bounded potentials, and compactness were developed by
Pinamonti, Squassina, and Vecchi \cite[Section~3]{PSV2019}.
We use the Hermitian norm on complex covectors in the definition of
total variation. Their convention sums the Euclidean norms of the real
and imaginary parts and uses $\dd-\ii A$; after changing the sign of
the potential, the two variation norms are equivalent. The Hermitian
convention makes multiplication by a unit complex phase an isometry.
The coercivity argument below combines this established compactness
mechanism with identification of the parallel kernel.

\begin{lemma}[Regularity of the parallel kernel]
\label{mag:parallel-kernel}
If $u\in L^1(M;\C)$ satisfies
\[
 Du+\ii Au\,\dd V_g=0
\]
in distributions, then $u$ is smooth up to the boundary. Its modulus
is constant on each connected component of $M$. In particular, on a
connected $M$, a nonzero such function exists if and only if
$A\in\mathcal B_M$.
\end{lemma}

\begin{proof}
The distributional equation first gives $u\in W^{1,1}(M)$.
The Sobolev chain rule shows that $\dd|u|=0$ almost everywhere:
where $u\ne0$,
\[
 \dd|u|=\operatorname{Re}
          \left(\frac{\overline u}{|u|}\dd u\right)
        =\operatorname{Re}(-\ii A|u|)=0,
\]
and the same weak identity holds across the zero set. Consequently
$|u|$ is constant on each component. There are finitely many components,
so $u$ is bounded. The original equation now implies
$u\in W^{1,\infty}(M)$. In a coordinate chart it reads
$\partial_j u=-\ii A_j u$. Differentiating this identity weakly gives
$u\in W^{k,p}$ for every integer $k$ and every finite $p$.
The argument applies in boundary charts as well. Sobolev embedding
then gives smoothness up to $\partial M$.

On a connected component where $u$ is nonzero, its constant modulus
$c$ is positive. The function $q=\overline u/c$ satisfies
$A=\dd q/(\ii q)$. The converse follows from
$\dd_A\overline q=0$.
\end{proof}

\begin{theorem}[Magnetic BV coercivity]
\label{mag:coercivity}
Suppose that $M$ is connected and $A\notin\mathcal B_M$. Then
$\kappa_A(M)>0$, and
\begin{equation}
 \label{mag:bv-poincare}
 \kappa_A(M)\|u\|_{L^1(M)}
       \le|\mathscr D_Au|(M)
 \qquad (u\in\BV(M;\C)).
\end{equation}
For a possibly disconnected $M$, positivity holds precisely when
$\dd_A$ has no nonzero global parallel section.
\end{theorem}

\begin{proof}
If $\kappa_A(M)=0$, choose $u_n\in\BV(M;\C)$ with
$\|u_n\|_1=1$ and $|\mathscr D_Au_n|(M)\to0$.
Equation \eqref{mag:variation-control} gives a uniform BV bound.
Passing to a subsequence, $u_n\to u$ in $L^1(M)$, with $\|u\|_1=1$.
Since $A$ is bounded, $Au_n\to Au$ in $L^1$, while the measures
$\mathscr D_Au_n$ converge to zero in total variation. Passing to
distributions yields $\mathscr D_Au=0$. Lemma
\ref{mag:parallel-kernel} produces a nonzero global parallel section.
On connected $M$, this contradicts $A\notin\mathcal B_M$.
The same compactness argument works on disconnected $M$ under the
stated kernel hypothesis. Conversely, a nonzero parallel section
makes the quotient defining $\kappa_A(M)$ vanish.
\end{proof}

\begin{proposition}[Gauge invariance and stability]
\label{mag:stability}
The number $\kappa_A(M)$ is gauge invariant. For any smooth real
one-forms $A,B$ on the same fixed Riemannian manifold,
\[
 |\kappa_A(M)-\kappa_B(M)|\le\|A-B\|_{L^\infty(M)}.
\]
\end{proposition}

\begin{proof}
For every BV function $u$, the triangle inequality for vector-valued
measures gives
\[
 \bigl||\mathscr D_Au|(M)-|\mathscr D_Bu|(M)\bigr|
       \le\|A-B\|_\infty\|u\|_1.
\]
Divide by $\|u\|_1$, take infima, and interchange $A$ and $B$.
If $b=\dd q/(\ii q)\in\mathcal B_M$, the BV product rule gives
$\mathscr D_A(qu)=q\mathscr D_{A+b}u$. Multiplication by $q$
preserves both the $L^1$ norm and total variation, and is a bijection
of $\BV(M;\C)$. Hence $\kappa_{A+b}(M)=\kappa_A(M)$.
\end{proof}

\subsection{Extension across internal cuts}

\begin{lemma}[The cost of an internal interface]
\label{mag:extension}
Let $D\subset M$ be an admissible smooth domain and
$\tau\in C^\infty(D,S^1)$ have finite magnetic energy.
Extend $\tau$ by zero to $M\setminus D$, writing
$u=\mathbf1_D\tau$. Then $u\in\BV(M;\C)$,
$\|u\|_1=\operatorname{Vol}(D)$, and
\begin{equation}
 \label{mag:extension-cost}
 |\mathscr D_Au|(M)
       \le \int_D|\dd_A\tau|\,\dd V_g+2P_M(D).
\end{equation}
Moreover,
\begin{equation}
 \label{mag:external-trace}
 \|\Tr u\|_{L^1(\partial M)}=E_M(D).
\end{equation}
For an ordinary one-sided domain, the coefficient $2$ in
\eqref{mag:extension-cost} improves to $1$.
\end{lemma}

\begin{proof}
Since $|\tau|=1$,
\[
 \int_D|\dd\tau|
       \le\int_D|\dd_A\tau|+\|A\|_\infty\operatorname{Vol}(D).
\]
Thus $\tau$ has integrable weak derivative on each side of every
smooth internal boundary chart. Its one-sided Sobolev traces have
modulus one. The BV gluing formula decomposes $Du$ into its
absolutely continuous derivative on $D$ and its jumps across the
internal hypersurfaces \cite{AFP}.

At a boundary point with $D$ on one side, the other value of the
extension is zero, so the jump has modulus one. If $D$ occupies both
sides of an internal hypersurface, the traces are two unit complex
numbers and their difference has modulus at most two. There is no
additional derivative on the exterior boundary of $M$.
The interior part of $\mathscr D_Au$ is
$\mathbf1_D\dd_A\tau\,\dd V_g$, giving
\eqref{mag:extension-cost}. The BV trace commutes with the Lipschitz
modulus, so $|\Tr u|=\Tr\mathbf1_D$. For the smooth domains under
consideration this characteristic trace is one almost everywhere on
the external boundary occupied by $D$, and zero elsewhere. This proves
\eqref{mag:external-trace}.
\end{proof}

Combining Theorem \ref{mag:coercivity} with Lemma
\ref{mag:extension}, and taking the infimum over $\tau$, gives a
uniform estimate over all admissible domains:
\begin{equation}
 \label{mag:all-domains}
 \boxed{\quad
 \iota_A(D)+2P_M(D)\ge\kappa_A(M)\operatorname{Vol}(D).
 \quad}
\end{equation}
There are no restrictions on the topology of $D$ or the number of
its internal boundary components.

\begin{theorem}[Uniform frustration on nearly full domains]
\label{mag:conjecture}
Let $(M,g)$ be compact with smooth boundary and let
$A\notin\mathcal B_M$. There exist $\varepsilon,\delta>0$ such that
every nonempty admissible smooth domain $D\subset M$ satisfying
\[
 \operatorname{Vol}(D)\ge(1-\varepsilon)\operatorname{Vol}(M),
 \qquad P_M(D)\le\varepsilon
\]
also satisfies $\iota_A(D)\ge\delta$.
This proves the conjecture in \cite[Section 4.1]{CGHP}.
\end{theorem}

\begin{proof}
First suppose that $M$ is connected. Set $V=\operatorname{Vol}(M)$
and $\kappa=\kappa_A(M)>0$. Choose
\[
 \varepsilon=\min\left\{\frac12,\frac{\kappa V}{8}\right\},
 \qquad \delta=\frac{\kappa V}{4}.
\]
Then \eqref{mag:all-domains} gives
\[
 \iota_A(D)\ge\kappa\operatorname{Vol}(D)-2P_M(D)
       \ge\frac{\kappa V}{2}-2\varepsilon
       \ge\delta.
\]

If $M$ is disconnected, at least one connected component $M_0$
satisfies $A|_{M_0}\notin\mathcal B_{M_0}$; otherwise the gauges
on its finitely many components would define a global gauge.
Write $V_0=\operatorname{Vol}(M_0)$ and
$\kappa_0=\kappa_{A|_{M_0}}(M_0)>0$. Choose
\[
 \varepsilon=\min\left\{\frac{V_0}{2\operatorname{Vol}(M)},
                         \frac{\kappa_0V_0}{8}\right\},
 \qquad \delta=\frac{\kappa_0V_0}{4}.
\]
The volume condition implies
$\operatorname{Vol}(D\cap M_0)\ge V_0/2$.
Apply Lemma \ref{mag:extension} and Theorem \ref{mag:coercivity}
on $M_0$ to the restriction of any unit phase on $D$.
Its energy on $D$ is at least
$\kappa_0V_0/2-2P_M(D)\ge\delta$.
Taking the infimum over all such phases proves the assertion.
\end{proof}

\subsection{Positivity of the magnetic Cheeger constants}

\begin{theorem}
\label{mag:cheeger-positive}
Suppose that $M$ is connected, $\partial M\ne\varnothing$, and
$A\notin\mathcal B_M$. Let $\kappa=\kappa_A(M)$ and let $C_T>0$
be a constant in the BV trace inequality
\[
 \|\Tr u\|_{L^1(\partial M)}
       \le C_T\bigl(\|u\|_{L^1(M)}+|Du|(M)\bigr).
\]
Then
\begin{equation}
 \label{mag:cheeger-bounds}
 h_A(M)\ge\frac{\kappa}{2}>0,\qquad
 h'_A(M)\ge
 \frac{1}{2C_T\left(1+(1+\|A\|_\infty)/\kappa\right)}>0.
\end{equation}
In particular, the magnetic Cheeger--Jammes bound of
\cite[Theorem 2.5]{CGHP} has a strictly positive right-hand side.
\end{theorem}

\begin{proof}
Equation \eqref{mag:all-domains} implies
\[
 \iota_A(D)+P_M(D)
       \ge\frac12\bigl(\iota_A(D)+2P_M(D)\bigr)
       \ge\frac{\kappa}{2}\operatorname{Vol}(D),
\]
which proves the first assertion. Next,
\eqref{mag:variation-control} and \eqref{mag:bv-poincare} give
\[
 \|\Tr u\|_{L^1(\partial M)}
 \le C_T\left(1+\frac{1+\|A\|_\infty}{\kappa}\right)
                     |\mathscr D_Au|(M).
\]
Apply this to $u=\mathbf1_D\tau$. Lemma \ref{mag:extension}
and minimization over the phase yield
\[
 E_M(D)\le
 C_T\left(1+\frac{1+\|A\|_\infty}{\kappa}\right)
       \bigl(\iota_A(D)+2P_M(D)\bigr).
\]
The last factor is bounded above by
$2(\iota_A(D)+P_M(D))$. Divide by $E_M(D)>0$ and take the
infimum to obtain the second assertion.
\end{proof}

\begin{remark}
On a disconnected manifold, the conclusion of Theorem
\ref{mag:cheeger-positive} requires the absence of a nonzero parallel
section on every component. A gauge-trivial component makes
$h_A(M)=0$. The nearly full domain statement in Theorem
\ref{mag:conjecture} only requires one component on which the gauge
obstruction is present, and therefore retains the general hypothesis
$A\notin\mathcal B_M$.
\end{remark}

\begin{corollary}[Stability under perturbations of the connection]
\label{mag:uniform-family}
On a fixed connected $(M,g)$, suppose that $A_\varepsilon\to A_0$
in $L^\infty$ and $A_0\notin\mathcal B_M$. Then for all sufficiently
small $\varepsilon$,
\[
 \iota_{A_\varepsilon}(D)+2P_M(D)
       \ge\frac{\kappa_{A_0}(M)}2\operatorname{Vol}(D)
\]
simultaneously for every admissible smooth domain $D$.
Both magnetic Cheeger constants have uniform positive lower bounds.
\end{corollary}

\begin{proof}
Proposition \ref{mag:stability} gives
$\kappa_{A_\varepsilon}(M)\ge\kappa_{A_0}(M)/2$ for small
$\varepsilon$. Apply \eqref{mag:all-domains} and
\eqref{mag:cheeger-bounds}; the norms $\|A_\varepsilon\|_\infty$
are uniformly bounded.
\end{proof}

\begin{remark}[Unitary bundles]
The compactness argument also applies to a smooth Hermitian vector
bundle with a unitary connection $\nabla$. Define BV sections in
smooth unitary trivializations and let $D_\nabla u$ be their
covariant distributional derivative. If the connection has no
nonzero global parallel section, then
\[
 \|u\|_{L^1(M)}\le C\,|D_\nabla u|(M).
\]
Indeed, normalized sections with covariant variation tending to
zero have a strongly convergent $L^1$ subsequence in a fixed finite
trivializing cover. The limit is distributionally parallel; the
argument of Lemma \ref{mag:parallel-kernel}, using metric
compatibility to obtain a constant norm, makes it a smooth nonzero
parallel section. This contradicts the kernel hypothesis.
Thus both curvature and flat holonomy can provide the obstruction
needed by the same coercivity estimate.
\end{remark}

\section{Weighted harmonic response and contact rigidity}

Let $(M,g)$ be a closed, connected, oriented smooth Riemannian manifold.
For a positive smooth function $\rho$, equip differential forms of every
degree with the inner product
\[
 (a,b)_\rho=\int_M\rho\langle a,b\rangle_g\,\dd V_g.
\]
The weighted adjoint and Hodge Laplacian are
\begin{equation}
 \dd^*_\rho=\rho^{-1}\dd^*(\rho\,\cdot)
 =\dd^*-\iota_{\nabla\log\rho},
 \qquad
 \Delta_\rho=\dd\dd^*_\rho+\dd^*_\rho\dd.
 \label{resp:weighted-adjoint}
\end{equation}
Their principal symbols give
$\sigma_2(\Delta_\rho)(x,\xi)=|\xi|_g^2\operatorname{id}$.
The nonnegative self-adjoint realization of $\Delta_\rho$ on weighted
$L^2$ has domain $H^2$ and compact resolvent. Its kernel consists of
smooth forms, and the weighted Hodge decomposition assigns to each
$c\in H^k_{\mathrm{dR}}(M;\R)$ a unique representative $h_\rho(c)$
satisfying $\dd h_\rho(c)=\dd^*_\rho h_\rho(c)=0$.
This is the compact weighted Hodge decomposition described by
Bueler and Prokhorenkov \cite[Section~1]{BuelerProkhorenkov2002}.
Its relation to Witten's deformed de Rham complex is explicit.
For $\phi=-\tfrac12\log\rho$, multiplication
$U_\rho a=\rho^{1/2}a$ is unitary from weighted to ordinary $L^2$, and
\[
 U_\rho\dd U_\rho^{-1}=e^{-\phi}\dd e^\phi
 =\dd+\dd\phi\wedge.
\]
The weighted adjoint and Laplacian transform accordingly into Witten's
operators \cite[Section~2, (11)--(12)]{Witten1982}.

Two geometric weights arise from a principal bundle with compact
structure group, an invariant inner product $B$ on its Lie algebra,
a connection $A$, and a nonzero parallel adjoint section $u$.
Since $\dd B(u,u)=2B(\nabla_Au,u)=0$, the line
$L=\R u\subset\operatorname{ad}P$ has a global parallel frame.
The induced de Rham complex satisfies
$\dd_L(a\otimes u)=\dd a\otimes u$. With the background metric fixed,
one can take
\begin{equation}
 \rho_A=1+B(u,u),\qquad
 \rho_B=1+|F_A|_{g,B}^2+|\nabla_A F_A|_{g,B}^2.
\end{equation}
Here the derivative of $F_A$ uses the tensor product of the
Levi--Civita and adjoint connections. Both weights are gauge invariant;
$\rho_A$ is constant, while $\rho_B$ records the curvature and its
covariant derivative. A normalized parallel frame identifies their
weighted Hodge theories with the scalar construction above.

\subsection{Variational response}

Define the energy of a cohomology class by
\begin{equation}
 E_\rho(c)=\|h_\rho(c)\|_\rho^2
 =\min_{\dd a=0,\,[a]=c}\int_M\rho|a|_g^2\,\dd V_g.
 \label{resp:energy}
\end{equation}
The space $\mathcal B^k$ of $L^2$ exact forms is closed by Hodge
theory. If $P_\rho$ is the weighted orthogonal projection onto
$\mathcal B^k$, the minimizing representative is characterized by
$(h_\rho(c),e)_\rho=0$ for every $e\in\mathcal B^k$.
The minimum principle belongs to the classical variational theory of
effective operators \cite[Section~13.1]{Milton2002}. A Hilbert-space
formulation, including the concavity of \eqref{resp:energy}, is given
in \cite[Theorem~6 and Corollary~8]{Welters2025}.
The response identities below specify its dependence on harmonic
representatives.

\begin{proposition}\label{resp:variation}
For an affine path $\rho_t=\rho_0+tq>0$, write $h_t=h_{\rho_t}(c)$.
Then
\begin{align}
 \dot h_t&=-P_{\rho_t}\left(\frac q{\rho_t}h_t\right),
 \label{resp:derivative}\\
 \frac{\dd}{\dd t}E_{\rho_t}(c)
 &=\int_M q|h_t|^2\,\dd V_g,\qquad
 \frac{\dd^2}{\dd t^2}E_{\rho_t}(c)
 =-2\|\dot h_t\|_{\rho_t}^2.
\end{align}
Equality in concavity at an interior point of a segment holds precisely
when its endpoint harmonic representatives coincide.
\end{proposition}

\begin{proof}
On the fixed closed space $\mathcal B^k$, the normal equation has a
bounded positive invertible coefficient operator. Its inverse depends
smoothly on $t$ on each compact interval of positive weights.
Elliptic regularity gives smooth dependence in every Sobolev norm.
Differentiating the normal equation, and noting that $\dot h_t$ is exact,
gives
\[
 (\dot h_t,e)_{\rho_t}=-\int_Mq\langle h_t,e\rangle\,\dd V_g.
\]
This proves \eqref{resp:derivative}. Differentiating the minimum energy
and using orthogonality gives its first derivative. Substituting
$e=\dot h_t$ gives the second derivative. For the equality assertion,
insert the minimizer at the interior weight into both endpoint
functionals. Equality in their weighted sum forces it to minimize
each functional, and uniqueness identifies the representatives.
The converse follows by using the common minimizer throughout the segment.
\end{proof}

Fix classes $c_1,\ldots,c_b$ and put
$T(\rho)=\sum_{j=1}^b E_\rho(c_j)$.
For a reference weight $\rho_0$, let $h_j=h_{\rho_0}(c_j)$ and
$\sigma_0=\sum_j|h_j|^2$. Weighted orthogonality yields
\begin{equation}
 \mathcal D_{\rho_0}(\rho)
 :=\int_M\rho\sigma_0\,\dd V_g-T(\rho)
 =\sum_{j=1}^b\|h_j-h_\rho(c_j)\|_\rho^2.
 \label{resp:defect}
\end{equation}
The left side is the difference between the affine tangent functional
at $\rho_0$ and $T(\rho)$. This identity is the Hodge-space specialization
of the classical trial-energy defect formula
\cite[Section~13.1, (13.9)--(13.10)]{Milton2002}.

\subsection{Recovery from bracket-generating harmonic fields}

From now on, the classes have degree one. Set $X_j=h_j^{\sharp_g}$.
Equation \eqref{resp:weighted-adjoint} shows that $h_j$ is also
$\rho$-harmonic exactly when
\begin{equation}
 X_j\log(\rho/\rho_0)=0.
\end{equation}
If the $X_j$ generate $TM$ by iterated brackets, every bracket also
annihilates this logarithm. Its differential therefore vanishes.
Consequently the labeled harmonic representatives determine the
weight up to a positive constant factor. Equal total masses remove
that factor. The following estimate quantifies this conclusion.

\begin{theorem}\label{resp:bracket-estimate}
Suppose the vector fields $X_j$ and their brackets of length at most
$r$ span $T_xM$ at every $x\in M$. Write
$\dd\mu=\rho_0\dd V_g$, $a=\rho/\rho_0$, and
$a_\mu=\mu(M)^{-1}\int_M a\,\dd\mu$.
There is a constant $C$, depending only on the reference data, such that
every positive smooth $\rho$ satisfies
\begin{equation}
 \|a-a_\mu\|_{H^{-1+1/r}(M)}^2
 \le C\|a\|_\infty\mathcal D_{\rho_0}(\rho).
 \label{resp:negative-stability}
\end{equation}
The Sobolev spaces use the fixed metric $g$.
\end{theorem}

\begin{proof}
The bracket-generating estimate of Rothschild--Stein
\cite[Theorem~17 and (17.20)]{RothschildStein1976}, with $p=2$,
gives the following estimate after a finite partition of unity.
The cutoff commutators contribute only to the $L^2$ remainder:
\[
 \|v\|_{H^{1/r}}
 \le C\left(\sum_j\|X_jv\|_2+\|v\|_2\right).
\]
Let $\Lambda=(1+\Delta_g)^{1/2}$ on functions. Apply this estimate
to $\Lambda^{-1}f$. Each commutator $[X_j,\Lambda^{-1}]$ has
pseudodifferential order $-1$, so
\begin{equation}
 \|f\|_{H^{-1+1/r}}
 \le C\left(\sum_j\|X_jf\|_{H^{-1}}+\|f\|_{H^{-1}}\right).
 \label{resp:shifted-estimate}
\end{equation}
On the subspace of zero $\mu$-mean the last term can be absorbed.
Indeed, failure would give $f_n$ of unit $H^{-1+1/r}$ norm and
zero mean with $X_jf_n\to0$ in $H^{-1}$. Compactness of
$H^{-1+1/r}\hookrightarrow H^{-1}$, followed by
\eqref{resp:shifted-estimate} on differences, gives convergence in
$H^{-1+1/r}$. The limit is annihilated by every $X_j$, hence by
their brackets, in the distributional sense. It is constant and
has zero mean, contradicting its unit norm. Thus
\begin{equation}
 \|f-f_\mu\|_{H^{-1+1/r}}
 \le C\sum_j\|X_jf\|_{H^{-1}}.
 \label{resp:negative-poincare}
\end{equation}

Set $e_j=h_j-h_\rho(c_j)$. The weighted product rule and the
two coclosedness equations give
\[
 \dd^*_{\rho_0}(a e_j)=-X_ja.
\]
The fixed operator $\dd^*_{\rho_0}$ maps $L^2$ to $H^{-1}$, while
\[
 \sum_j\|a e_j\|_{L^2(\mu)}^2
 \le\|a\|_\infty\sum_j\int_M a|e_j|^2\,\dd\mu
 =\|a\|_\infty\mathcal D_{\rho_0}(\rho).
\]
Combining these bounds with \eqref{resp:negative-poincare} proves
\eqref{resp:negative-stability}.
\end{proof}

The constant in \eqref{resp:negative-stability} is independent of the
minimum of the varying weight. Each positive weight has its own
harmonic representatives; the estimate uses the fixed reference
operators and the upper bound for its ratio to $\rho_0$.

\subsection{The Heisenberg manifold and the optimal exponent}

Let $H$ be the Heisenberg group with multiplication
\[
 (x,y,z)(x',y',z')=(x+x',y+y',z+z'+xy'),
\]
and let $N=\Gamma\backslash H$ for its integer lattice $\Gamma$.
The invariant coframe and metric
\[
 \alpha=\dd x,\qquad \beta=\dd y,\qquad
 \vartheta=\dd z-x\,\dd y,\qquad
 g=\alpha^2+\beta^2+\vartheta^2
\]
descend to $N$, whose volume is one. The dual fields are
$X=\partial_x$, $Y=\partial_y+x\partial_z$, and $Z=\partial_z$,
with $[X,Y]=Z$. They preserve the volume, so $\alpha$ and $\beta$
are harmonic. Their classes form a basis of $H^1(N;\R)$, as follows
from the abelianization of $\Gamma$. The plane field
$\ker\vartheta=\operatorname{span}(X,Y)$ is contact since
$\dd\vartheta=-\alpha\wedge\beta$.

\begin{theorem}\label{resp:sharp-contact}
For $T(\rho)=E_\rho([\alpha])+E_\rho([\beta])$ and every
positive smooth weight of mass one,
\[
 T(\rho)\le2,\qquad T(\rho)=2\ \Longleftrightarrow\ \rho=1.
\]
There is $c>0$ such that, whenever $0<\rho\le R$,
\begin{equation}
 2-T(\rho)\ge\frac cR\|\rho-1\|_{H^{-1/2}(N)}^2.
 \label{resp:contact-deficit}
\end{equation}
For each fixed $R>1$ and $s<1/2$, no positive constant can replace
the right side by $c_s\|\rho-1\|_{H^{-s}}^2$ for all such weights.
This failure also holds in every open smooth neighborhood of $1$.
\end{theorem}

\begin{proof}
Here $\sigma_0=2$ at $\rho_0=1$, so
\eqref{resp:defect} proves the upper bound. Equality gives
$X\rho=Y\rho=0$, hence $Z\rho=0$, and the mass fixes the constant.
Theorem~\ref{resp:bracket-estimate} with $r=2$ gives
\eqref{resp:contact-deficit}.

For optimality, put $\lambda_k=2\pi k$, $k\ge1$, and define
\begin{equation}
 \Psi_k(x,y,z)=(2k)^{1/4}e^{2\pi\ii kz}
 \sum_{n\in\Z}e^{-\pi k(x+n)^2}e^{2\pi\ii kny}.
\end{equation}
The series and its derivatives converge locally uniformly.
Reindexing under $(x,y,z)\mapsto(x+1,y,z+y)$ proves invariance
under the first lattice generator; invariance under unit translations
in $y$ and $z$ is immediate. Orthogonality in $y$ gives
\[
 \|\Psi_k\|_2^2
 =(2k)^{1/2}\int_\R e^{-2\pi k t^2}\,\dd t=1.
\]
With $L=-X^2-Y^2$ and $\Delta=L-Z^2$, direct differentiation yields
\[
 Z\Psi_k=\ii\lambda_k\Psi_k,\quad
 Y\Psi_k=-\ii X\Psi_k,\quad
 L\Psi_k=\lambda_k\Psi_k.
\]
Moreover $[L,X]=2ZY$ and $[L,Y]=-2ZX$. Therefore
\begin{align*}
 \Delta\Psi_k&=(\lambda_k^2+\lambda_k)\Psi_k,\\
 \Delta(X\Psi_k)&=(\lambda_k^2+3\lambda_k)X\Psi_k,\\
 \Delta(Y\Psi_k)&=(\lambda_k^2+3\lambda_k)Y\Psi_k.
\end{align*}
Set $F_k=\sqrt2\operatorname{Re}\Psi_k$. Integration in $z$ gives
$\int_NF_k=0$ and $\|F_k\|_2=1$. The real Laplacian identities
pass to $F_k$, and
$\|XF_k\|_2^2+\|YF_k\|_2^2=\lambda_k$.

For each fixed $k$, consider $\rho_t=1+tF_k$ for sufficiently small
$t$. Differentiate the two harmonic representatives at $t=0$ and
write their exact derivatives as $\dd\phi_X$ and $\dd\phi_Y$.
Equation \eqref{resp:weighted-adjoint} gives
$\Delta\phi_X=XF_k$ and $\Delta\phi_Y=YF_k$, with zero-mean
potentials. The eigenvalue identities imply
\[
 \|\dd\phi_X\|_2^2+\|\dd\phi_Y\|_2^2
 =\frac{\lambda_k}{\lambda_k^2+3\lambda_k}
 =\frac1{\lambda_k+3}.
\]
Proposition~\ref{resp:variation} consequently gives
\begin{equation}
 -D^2T_1[F_k,F_k]=\frac2{2\pi k+3},\qquad
 2-T(1+tF_k)=\frac{t^2}{\lambda_k+3}+o_k(t^2).
 \label{resp:exact-hessian}
\end{equation}
In the spectral convention
$\|f\|_{H^a}^2=\langle(1+\Delta)^af,f\rangle$, we have
$\|F_k\|_{H^{-s}}^2=(1+\lambda_k+\lambda_k^2)^{-s}$.
If the asserted stronger estimate held, first taking $t\to0$ for
each fixed $k$ would give
\[
 \frac1{\lambda_k+3}
 \ge c_s(1+\lambda_k+\lambda_k^2)^{-s}.
\]
Letting $k\to\infty$ contradicts $s<1/2$. For each fixed $k$,
the path lies in $0<\rho_t\le R$ and in any prescribed open smooth
neighborhood of $1$ once $t$ is sufficiently small. The argument
therefore needs only the fixed-$k$ expansion
\eqref{resp:exact-hessian}.
\end{proof}

\section{Two metrics on a circle bundle}

Let $(\Sigma,h)$ be a closed, connected, oriented Riemannian surface, with area
form $\dd A_h$. Let $\pi:P\to\Sigma$ be a principal circle bundle whose
fundamental field $X$ has period $2\pi$. A real connection $\alpha$ satisfies
\[
 \alpha(X)=1,\qquad \mathcal L_X\alpha=0,\qquad
 \dd\alpha=\pi^*F,\qquad F=f\,\dd A_h.
\]
Throughout this section $f$ is smooth and positive. Put
\[
 I=\int_\Sigma f\,\dd A_h=2\pi e>0,
 \qquad \bar f=I/|\Sigma|_h,
\]
where $e$ is the Euler number in this curvature convention. We compare
\begin{equation}\label{geo:metrics}
 g_A=\alpha^2+\pi^*(\bar f h),\qquad
 g_B=\alpha^2+\pi^*(fh).
\end{equation}
Both metrics have volume $2\pi I=4\pi^2e$. Orient $P$ by
$\alpha\wedge\pi^*\dd A_h$. For $g_B$ one has
\begin{equation}\label{geo:adapted}
 *_B\dd\alpha=\alpha,\qquad *_B\alpha=\dd\alpha,
 \qquad |X|_{g_B}=1.
\end{equation}
The curvature therefore fixes the horizontal conformal factor of $g_B$.
The smooth spectral mechanism for these connection metrics is contained in
Nagy's circle-bundle analysis~\cite[Theorem~2.1]{Nagy1999}. We give the identities needed
for estimates that remain uniform when $f$ approaches zero.

\subsection{The invariant complex and its Hodge pairing}

Every invariant $k$-form has a unique expression
$\pi^*\beta+\alpha\wedge\pi^*\gamma$, where
$(\beta,\gamma)\in\Omega^k(\Sigma)\oplus\Omega^{k-1}(\Sigma)$.
Exterior differentiation becomes
\begin{equation}
 D_F(\beta,\gamma)=(\dd\beta+F\wedge\gamma,-\dd\gamma).
\end{equation}
Since $\dd F=0$, this differential satisfies $D_F^2=0$.
For $g_w=\alpha^2+\pi^*(wh)$, with any smooth positive $w$, its exact norm is
\begin{equation}\label{geo:norm}
 \|\pi^*\beta+\alpha\wedge\pi^*\gamma\|_{g_w}^2
 =2\pi\int_\Sigma
 \bigl(w^{1-k}|\beta|_h^2+w^{2-k}|\gamma|_h^2\bigr)\,\dd A_h.
\end{equation}
Terms in nonexistent degrees are omitted.

The projection $(P,g_w)\to(\Sigma,wh)$ is a Riemannian submersion
with geodesic unit circle fibers. Preservation of harmonic one-forms
under this submersion follows from Watson
\cite[Theorem~3.3]{Watson1975}, together with conformal invariance
of the Hodge star on one-forms on a surface. The next calculation
identifies the full harmonic spaces and their cohomological pairings.

\begin{proposition}\label{geo:harmonic}
The middle-degree harmonic spaces are
\[
 \cH^1(P,g_w)=\pi^*\cH^1(\Sigma,h),\qquad
 \cH^2(P,g_w)=\alpha\wedge\pi^*\cH^1(\Sigma,h).
\]
Under pullback in degree one and normalized fiber integration in degree two,
their cohomological Hodge pairings equal $2\pi$ times the Hodge pairing on
$H^1(\Sigma;\R)$. Consequently $g_A$ and $g_B$ induce the same Hodge inner
products on all real de Rham cohomology groups.
\end{proposition}

\begin{proof}
Every circle translation is an isometry isotopic to the identity. Uniqueness
of harmonic representatives implies that all harmonic forms are invariant.
Write a harmonic one-form as $\pi^*\beta+v\alpha$. Closedness and
coclosedness give
\[
 \dd v=0,\qquad \dd\beta+vF=0,\qquad \delta_h\beta=0.
\]
Integrating the middle identity yields $vI=0$, hence $v=0$ and
$\beta\in\cH^1(\Sigma,h)$. The Hodge star gives the degree-two statement.
Equation~\eqref{geo:norm} gives both middle-degree pairings.
These identifications agree with cohomology: if $\alpha$ is changed, the
degree-two representative changes by a pullback of a base two-form, and
every such form has the form $cF+\dd\theta$. Its pullback is
$\dd(c\alpha+\pi^*\theta)$. Finally the norm of the constant class in
degree zero and that of the integral-one class in degree three are the
volume and its reciprocal. The volumes in~\eqref{geo:metrics} agree.
\end{proof}

\subsection{A uniform gap in every degree}

Write $Q_k(\omega)=\|\dd\omega\|^2+\|\delta\omega\|^2$ for the Hodge
quadratic form of $g_B$.

\begin{lemma}\label{geo:fourier}
If $\mathcal L_X\omega=\ii m\omega$, $m\in\Z$, then
$Q_k(\omega)\ge m^2\|\omega\|^2$ in every degree.
\end{lemma}

\begin{proof}
Cartan's formula, adjointness, and Cauchy--Schwarz give
\[
 |m|\|\omega\|^2
 \le\bigl(\|\iota_X\omega\|^2+\|\alpha\wedge\omega\|^2\bigr)^{1/2}
 Q_k(\omega)^{1/2}
 =\|\omega\|Q_k(\omega)^{1/2}.
\]
The last equality uses $|X|=1$ and $X^\flat=\alpha$.
\end{proof}

Let $\lambda=\lambda_1(\Sigma,h)>0$. For $0<f\le C$, define
\begin{equation}
 \nu=\lambda/C,\qquad
 L(\nu)=\frac{1+2\nu-\sqrt{1+4\nu}}2,\qquad
 c_* =\min\{1,L(\nu)\}.
\end{equation}

\begin{theorem}\label{geo:gap}
For every smooth connection with curvature $f\,\dd A_h$, $0<f\le C$,
\[
 \operatorname{spec}(\Delta_k^{g_B})\subset\{0\}\cup[c_*,\infty),
 \qquad 0\le k\le3.
\]
The constant is independent of $\inf f$, derivatives of $f$, and the
connection within this curvature range. If $f\equiv C$ is compatible with
the Euler integral, the first positive eigenvalue in degrees one and two
equals $c_*$.
\end{theorem}

\begin{proof}
Nonzero Fourier modes are controlled by Lemma~\ref{geo:fourier}. For an
invariant function the norm and energy, divided by $2\pi$, are
$\int f|u|^2$ and $\int|\dd u|^2$. If $\bar u_f=I^{-1}\int fu$, the
minimizing property of the weighted mean and the ordinary Poincare
inequality imply
\begin{equation}\label{geo:poincare}
 \int_\Sigma f|u-\bar u_f|^2\,\dd A_h
 \le \nu^{-1}\int_\Sigma|\dd u|^2\,\dd A_h.
\end{equation}
Thus degrees zero and three have gap at least $\min\{1,\nu\}$.

For an invariant one-form $\omega=\pi^*\beta+v\alpha$,
\begin{align}
 (2\pi)^{-1}\|\omega\|^2
 &=\|\beta\|_h^2+\int_\Sigma f|v|^2\,\dd A_h,\\
 (2\pi)^{-1}Q_1(\omega)
 &=\int_\Sigma\left(
 f^{-1}|\dd\beta+vf\,\dd A_h|_h^2
 +|\dd v|_h^2+f^{-1}|\delta_h\beta|^2\right)\,\dd A_h.
 \label{geo:one-energy}
\end{align}
By~\eqref{geo:adapted}, $\Delta_1\alpha=\alpha$. Remove this eigenform
component and the harmonic component. Then
$\int fv=0$ and $\beta\perp\cH_h^1$. Set
\[
 s=\int f^{-1}|\dd\beta|^2,\qquad
 t=\int f^{-1}|\delta_h\beta|^2,\qquad z^2=\int f|v|^2.
\]
All unmarked integrals are over $(\Sigma,h)$. Equation~\eqref{geo:poincare}
gives $\int|\dd v|^2\ge\nu z^2$. The positive one-form gap of a closed
oriented surface equals its scalar gap, by the exact and coexact Hodge
decomposition. Hence $s+t\ge\nu\|\beta\|_h^2$. Expanding
\eqref{geo:one-energy} gives
\[
 (2\pi)^{-1}Q_1(\omega)
 \ge s+t+(1+\nu)z^2-2\sqrt s\,z.
\]
The smaller eigenvalue of
\begin{equation}\label{geo:matrix}
 \begin{pmatrix}\nu&-\sqrt\nu\\-\sqrt\nu&1+\nu\end{pmatrix}
\end{equation}
is $L(\nu)\le\nu$. Apply this matrix bound to $(\sqrt{s/\nu},z)$ and use
$t\ge L(\nu)t/\nu$. The result is
$Q_1(\omega)\ge L(\nu)\|\omega\|^2$. Adding the removed components and
using the Hodge star proves the bound in all degrees.

For sharpness at $f=C$, let $\Delta_h\phi=\lambda\phi$ and
$\|\phi\|_2=1$. The invariant subspace spanned by
$\lambda^{-1/2}\pi^*(*_h\dd\phi)$ and
$C^{-1/2}(\pi^*\phi)\alpha$ has norm matrix $2\pi I_2$ and Hodge
matrix~\eqref{geo:matrix}. These statements follow by polarization of
\eqref{geo:one-energy} and $\dd *_h\dd\phi=-\lambda\phi\,\dd A_h$;
integration by parts also proves invariance of the subspace. Thus
$L(\nu)$ is an eigenvalue. The form $\alpha$ supplies eigenvalue one.
\end{proof}

The quadratic relation underlying this spectral block also occurs in
Peralta-Salas and Slobodeanu \cite[Proposition~1, (2.10)]{PeraltaSalasSlobodeanu2021}.
Their unit Reeb field has curl eigenvalue two. Rescaling our metric to
$\widehat g=g_B/4$ and the contact form to $\widehat\alpha=\alpha/2$
gives that convention; curl eigenvalues multiply by two and scalar
Laplacian eigenvalues by four. At constant $f=C$, their relation becomes
$\nu=\mu(\mu-1)$, where $\mu$ is an eigenvalue of the curl operator
$*_{g_B}\dd$ on coexact one-forms.
The smaller root has square $L(\nu)$.
The inequalities above give the stated uniform bound for variable $f$.

\subsection{Realized degeneration}

Choose smooth $f_0\ge0$, vanishing on a nonempty open set, with integral
$I$. A connection $\alpha_0$ with curvature $f_0\,\dd A_h$ exists because
this form represents the prescribed Euler class. Choose a smooth one-form
$b$ with $\dd b=(\bar f-f_0)\dd A_h$, and put
\begin{equation}\label{geo:degeneration}
 f_\varepsilon=(1-\varepsilon)f_0+\varepsilon\bar f,
 \qquad \alpha_\varepsilon=\alpha_0+\varepsilon\pi^*b,
 \qquad 0<\varepsilon\le1.
\end{equation}
The adapted metrics converge as smooth tensors to
$\alpha_0^2+\pi^*(f_0h)$, which has rank one over the open zero set.
Their volumes and cohomological Hodge pairings remain fixed. Theorem
\ref{geo:gap} supplies a common positive Hodge gap, with
$C=\max\{\|f_0\|_\infty,\bar f\}$. The next section identifies the
scalar spectral limit through a closed quadratic form.

\section{The scalar limit and its magnetic boundary response}

More generally, assume smooth connections $\alpha_\varepsilon\to\alpha_0$
in $C^1$ on the fixed bundle, with curvature densities
$f_\varepsilon>0$ converging uniformly to smooth $f_0\ge0$ and
$\int f_0\,\dd A_h=I>0$. Let
$g_\varepsilon=\alpha_\varepsilon^2+\pi^*(f_\varepsilon h)$.
In a local trivialization $\alpha_\varepsilon=\dd\theta+A_\varepsilon$,
the character $e^{\ii m\theta}$ identifies scalar functions with sections
of a fixed Hermitian line bundle $L^m$, with covariant derivative
$\nabla_{m,\varepsilon}=\dd-\ii m A_\varepsilon$.
This is the standard character and associated-line-bundle decomposition
for connection metrics; see Jung and Zelditch
\cite[Sections~2.2--3.3]{JungZelditch2024} for the identification and
the horizontal--vertical splitting of the scalar quadratic form.
Define on
$H^1(\Sigma;L^m)$
\begin{equation}
 a_{m,\varepsilon}(s,t)=\int_\Sigma
 \langle\nabla_{m,\varepsilon}s,\nabla_{m,\varepsilon}t\rangle_h\,\dd A_h,
 \qquad b_\varepsilon(s,t)=\int_\Sigma
 f_\varepsilon\langle s,t\rangle\,\dd A_h.
\end{equation}
Fiber integration gives the norm $2\pi b_\varepsilon(s,s)$ and scalar
energy $2\pi(a_{m,\varepsilon}(s,s)+m^2b_\varepsilon(s,s))$.
The cancellation of the horizontal conformal factor uses
$\dim\Sigma=2$.

\begin{theorem}\label{lim:scalar}
Every fixed scalar eigenvalue in character $m$ converges, with multiplicity,
to $m^2+\nu_{m,j}$, where the finite generalized eigenvalues are defined by
\begin{equation}\label{lim:generalized}
 a_{m,0}(s,t)=\nu_{m,j}b_0(s,t),\qquad t\in H^1(\Sigma;L^m).
\end{equation}
In each bounded spectral window whose endpoints avoid the limiting
spectrum, the full scalar spectrum converges to the union of these
character spectra, preserving multiplicities.
\end{theorem}

\begin{proof}
Fix $m$. The form $q_0=a_{m,0}+b_0$ defines a norm equivalent to $H^1$.
Indeed, a sequence of unit $H^1$ norm with $q_0\to0$ has an $L^2$-convergent
subsequence. Its covariant derivatives tend to zero. In local
trivializations this implies strong $H^1$ convergence to a parallel section
whose $b_0$ norm is zero. A parallel section has constant norm; $I>0$
forces this section to vanish, a contradiction.

On this fixed Hilbert space, let $T_\varepsilon$ represent
$a_{m,\varepsilon}+b_\varepsilon$ and $B_\varepsilon$ represent
$b_\varepsilon$. Uniform coefficient convergence gives
$T_\varepsilon\to1$ and $B_\varepsilon\to B_0$ in operator norm.
The operators $B_\varepsilon$ are compact by the compact embedding
$H^1\hookrightarrow L^2$. Therefore
\[
 K_\varepsilon=T_\varepsilon^{-1/2}B_\varepsilon T_\varepsilon^{-1/2}
 \longrightarrow B_0
\]
in norm. Their positive eigenvalues are
$\theta_j(\varepsilon)=(1+\nu_{m,j}(\varepsilon))^{-1}$.
Since $f_0$ is positive on a nonempty open set, $B_0$ has infinite rank.
Every fixed ordered positive eigenvalue therefore has a positive limit.
Compact self-adjoint perturbation theory
\cite[Chapter~IV, Section~3]{Kato1995} proves the asserted convergence
and convergence of each finite spectral cluster. This also excludes
additional bounded eigenvalue sequences: they correspond to eigenvalues
of $K_\varepsilon$ bounded away from zero. Finally the scalar energy gives
$\lambda\ge m^2$, so only finitely many characters occur in a bounded
window.
\end{proof}

\subsection{Eliminating a zero-density region}

Suppose $\Omega\subset\Sigma$ is a proper domain with smooth boundary,
$f_0=0$ on $\overline\Omega$, and $f_0>0$ in the interior of
$N=\Sigma\setminus\Omega$. For a boundary section $\varphi$, let
$H_m\varphi$ be its covariantly harmonic extension to $\Omega$. Dirichlet
minimization gives existence and uniqueness. Define
\begin{equation}
 \langle\Lambda_{m,\Omega}\varphi,\psi\rangle
 =\int_\Omega
 \langle\nabla_{m,0}H_m\varphi,\nabla_{m,0}H_m\psi\rangle_h\,\dd A_h.
\end{equation}

\begin{proposition}
Problem~\eqref{lim:generalized} is equivalent to the generalized problem
on $H^1(N;L^m)$ with mass $\int_Nf_0\langle v,w\rangle\,\dd A_h$ and
energy
\begin{equation}\label{lim:effective}
 \mathfrak a_m(v,w)=\int_N
 \langle\nabla_{m,0}v,\nabla_{m,0}w\rangle_h\,\dd A_h
 +\langle\Lambda_{m,\Omega}\Tr v,\Tr w\rangle.
\end{equation}
This energy, together with the mass, defines a complete norm equivalent to
$H^1(N;L^m)$, with compact embedding into the weighted mass space.
\end{proposition}

\begin{proof}
Extend $v$ into $\Omega$ by $H_m\Tr v$. This global $H^1$ section preserves
the mass and minimizes the energy among extensions, yielding
\eqref{lim:effective}. Testing~\eqref{lim:generalized} with zero-trace
sections supported in $\Omega$ proves that each global eigenfunction has
this extension. Every global test section splits into a harmonic extension
and a zero-trace section on $\Omega$, proving the converse. Norm equivalence
follows from the coercivity of $q_0$ and bounded trace extension; compactness
follows from the ordinary Rellich theorem on $N$.
\end{proof}

For smooth solutions the interface condition is
$\nabla_{n_N}v+\Lambda_{m,\Omega}\Tr v=0$, with $n_N$ the outward normal
of $N$. Thus the limiting boundary energy retains the connection on the
zero-density region.

\subsection{Flat holonomy in the limiting boundary operator}

On the flat cylinder $\Omega=[0,\ell]\times\R/(2\pi\Z)$, let
$h=\dd r^2+\dd y^2$ and $A_0=\vartheta\,\dd y$. For angular mode
$e^{\ii ky}$ set $a=|k-m\vartheta|$. Solving $-u''+a^2u=0$ with prescribed
endpoint values gives the matrix
\begin{equation}
 \Lambda_{m,k}=a
 \begin{pmatrix}
 \coth(a\ell)&-\operatorname{csch}(a\ell)\\
 -\operatorname{csch}(a\ell)&\coth(a\ell)
 \end{pmatrix}.
\end{equation}
At $a=0$ its value is
$\ell^{-1}\left(\begin{smallmatrix}1&-1\\-1&1\end{smallmatrix}\right)$.
The two eigenvalues are $a\tanh(a\ell/2)$ and $a\coth(a\ell/2)$, with
their continuous values at zero. These branches agree with
\cite[Examples~2.2--2.3]{CekicSiffert} and
\cite[Theorem~12]{ProvenzanoSavo}, after accounting for the cylinder
length and boundary-length normalization. Consequently
\[
 \dim\ker\Lambda_{m,\Omega}
 =\begin{cases}1,&m\vartheta\in\Z,\\0,&m\vartheta\notin\Z.\end{cases}
\]
For $m\vartheta\notin\Z$, the potential $-m\vartheta\,\dd y$ cannot be
gauged away. The magnetic BV coercivity theorem therefore supplies uniform
frustration bounds on varying subdomains of this flat region.
Corollary~\ref{mag:uniform-family} makes these bounds uniform for nearby
flat potentials as well.

This cylinder occurs in the closed-surface construction by choosing a
nonseparating annulus in a positive-genus surface, taking $f_0$ to vanish
on its closure, and choosing $h$ to be cylindrical there. Adding a global
closed one-form with nonzero period around its core varies $\vartheta$
while leaving all curvature densities unchanged. Formula
\eqref{geo:degeneration} supplies positive approximations for every such
connection. The cohomological Hodge pairings remain fixed, while the
limiting boundary operator distinguishes trivial and nontrivial flat
holonomy.


\begin{thebibliography}{99}

\bibitem{LLPP2015}
C.~Lange, S.~Liu, N.~Peyerimhoff, and O.~Post,
Frustration index and Cheeger inequalities for discrete and continuous
magnetic Laplacians,
\emph{Calc. Var. Partial Differential Equations} \textbf{54} (2015),
4165--4196.
\href{https://doi.org/10.1007/s00526-015-0935-x}{doi:10.1007/s00526-015-0935-x}.

\bibitem{Jammes2015}
P.~Jammes,
Une in\'egalit\'e de Cheeger pour le spectre de Steklov,
\emph{Ann. Inst. Fourier} \textbf{65} (2015), no.~3, 1381--1385.
\href{https://doi.org/10.5802/aif.2960}{doi:10.5802/aif.2960}.

\bibitem{CGHP}
T.~Chakradhar, K.~Gittins, G.~Habib, and N.~Peyerimhoff,
A note on the magnetic Steklov operator on functions,
\emph{Mathematika} \textbf{71} (2025), no.~4, e70037.
\href{https://doi.org/10.1112/mtk.70037}{doi:10.1112/mtk.70037}.

\bibitem{PSV2019}
A.~Pinamonti, M.~Squassina, and E.~Vecchi,
Magnetic BV-functions and the Bourgain--Brezis--Mironescu formula,
\emph{Adv. Calc. Var.} \textbf{12} (2019), no.~3, 225--252.
\href{https://doi.org/10.1515/acv-2017-0019}{doi:10.1515/acv-2017-0019}.

\bibitem{AFP}
L.~Ambrosio, N.~Fusco, and D.~Pallara,
\emph{Functions of Bounded Variation and Free Discontinuity Problems},
Oxford Mathematical Monographs, Clarendon Press, Oxford, 2000.
\href{https://doi.org/10.1093/oso/9780198502456.001.0001}{doi:10.1093/oso/9780198502456.001.0001}.

\bibitem{Milton2002}
G.~W.~Milton,
\emph{The Theory of Composites},
Cambridge Monographs on Applied and Computational Mathematics, vol.~6,
Cambridge University Press, Cambridge, 2002.
\href{https://doi.org/10.1017/CBO9780511613357}{doi:10.1017/\allowbreak CBO9780511613357}.

\bibitem{Welters2025}
A.~Welters,
Effective operators in the theory of composites: Hilbert space framework,
\href{https://arxiv.org/abs/2512.09860}{arXiv:2512.09860}, 2025.

\bibitem{RothschildStein1976}
L.~P.~Rothschild and E.~M.~Stein,
Hypoelliptic differential operators and nilpotent groups,
\emph{Acta Math.} \textbf{137} (1976), 247--320.
\href{https://doi.org/10.1007/BF02392419}{doi:10.1007/BF02392419}.

\bibitem{Nagy1999}
P.-A.~Nagy,
Le spectre du laplacien agissant sur les $p$-formes diff\'erentielles
d'un fibr\'e en cercle,
\emph{S\'eminaire de th\'eorie spectrale et g\'eom\'etrie}
\textbf{17} (1998--1999), 185--195.
\href{https://www.numdam.org/item/TSG_1998-1999__17__185_0/}{Numdam}.

\bibitem{CekicSiffert}
M.~Ceki\'c and A.~Siffert,
Magnetic Steklov problem on surfaces,
\emph{J. Funct. Anal.} \textbf{289} (2025), no.~12, 111159.
\href{https://doi.org/10.1016/j.jfa.2025.111159}{doi:10.1016/j.jfa.2025.111159}.

\bibitem{ProvenzanoSavo}
L.~Provenzano and A.~Savo,
Geometry of the magnetic Steklov problem on Riemannian annuli,
\emph{Commun. Contemp. Math.} \textbf{27} (2025), no.~7, 2550001.
\href{https://doi.org/10.1142/S0219199725500014}{doi:10.1142/S0219199725500014}.

\bibitem{Shigekawa1987}
I.~Shigekawa,
Eigenvalue problems for the Schr\"odinger operator with the magnetic field
on a compact Riemannian manifold,
\emph{J. Funct. Anal.} \textbf{75} (1987), no.~1, 92--127.
\href{https://doi.org/10.1016/0022-1236(87)90108-X}{doi:10.1016/0022-1236(87)90108-X}.

\bibitem{BuelerProkhorenkov2002}
E.~Bueler and I.~Prokhorenkov,
Hodge theory and cohomology with compact supports,
\emph{Soochow J. Math.} \textbf{28} (2002), no.~1, 33--55.
\href{https://arxiv.org/abs/math/9804152}{arXiv:math/9804152}.

\bibitem{Witten1982}
E.~Witten,
Supersymmetry and Morse theory,
\emph{J. Differential Geom.} \textbf{17} (1982), no.~4, 661--692.
\href{https://doi.org/10.4310/jdg/1214437492}{doi:10.4310/jdg/1214437492}.

\bibitem{Watson1975}
B.~Watson,
$\delta$-commuting mappings and Betti numbers,
\emph{T\^ohoku Math. J. (2)} \textbf{27} (1975), no.~2, 135--152.
\href{https://doi.org/10.2748/tmj/1178240982}{doi:10.2748/tmj/1178240982}.

\bibitem{PeraltaSalasSlobodeanu2021}
D.~Peralta-Salas and R.~Slobodeanu,
Energy minimizing Beltrami fields on Sasakian $3$-manifolds,
\emph{Int. Math. Res. Not. IMRN} \textbf{2021}, no.~9, 6656--6690.
\href{https://doi.org/10.1093/imrn/rnz044}{doi:10.1093/imrn/rnz044}.

\bibitem{JungZelditch2024}
J.~Jung and S.~Zelditch,
$2$-nodal domain theorems for higher-dimensional circle bundles,
\emph{J. Spectr. Theory} \textbf{14} (2024), no.~4, 1451--1474.
\href{https://doi.org/10.4171/JST/530}{doi:10.4171/JST/530}.

\bibitem{Kato1995}
T.~Kato,
\emph{Perturbation Theory for Linear Operators},
Classics in Mathematics, Springer, Berlin, 1995,
reprint of the corrected second edition (1980).
\href{https://doi.org/10.1007/978-3-642-66282-9}{doi:10.1007/978-3-642-66282-9}.

\end{thebibliography}
\end{document}